\documentclass{article}
\usepackage[english]{babel}
\usepackage[utf8]{inputenc}
\usepackage[T1]{fontenc}
\usepackage[normalem]{ulem}
\usepackage{amssymb}
\usepackage{amsmath,amsfonts}
\usepackage[top=3cm,bottom=3cm,left=2.5cm,right=2.5cm]{geometry}
\usepackage{amssymb}
\usepackage[shortlabels]{enumitem}
\usepackage[breaklinks]{hyperref}
\usepackage{graphicx} 
\usepackage{amsthm}

\usepackage{authblk}
\usepackage{tikz}
\usepackage{caption}
\usepackage{subcaption}
\usepackage{cleveref}

\theoremstyle{plain}
\newtheorem{theorem}{Theorem}[section]
\newtheorem{thm}{Theorem}

\newtheorem{proposition}[theorem]{Proposition}
\newtheorem{lemma}[theorem]{Lemma}

\theoremstyle{definition}

\newcommand{\cA}{\mathcal{A}}

\newcommand{\cQ}{\mathcal{Q}}

\newcommand{\cT}{\mathcal{T}}

\newcommand{\eps}{\varepsilon}

\newcommand{\vol}{\mathrm{Vol}}
\newcommand{\diam}{\mathrm{diam}}

\title{Spanning trees of bounded degree in random geometric graphs}
\author[1]{Michael Anastos}
\author[2]{Sahar Diskin}
\author[3]{Dawid Ignasiak}
\author[3]{Lyuben Lichev}
\author[4]{Yetong Sha}
\affil[1]{Institute of Science and Technology Austria (ISTA), 3400 Klosterneurburg, Austria}
\affil[2]{School of Mathematical Sciences, Tel Aviv University, 6997801 Tel Aviv, Israel}
\affil[3]{Institute of Statistics and Mathematical Methods in Economics, Technical University of Vienna, A-1040 Vienna, Austria}
\affil[4]{Institute for Advanced Study in Mathematics, Harbin Institute of Technology, 150001 Harbin, China}

\date{\today}

\begin{document}

\maketitle

\begin{abstract}
We determine the sharp threshold for the containment of all $n$-vertex trees of bounded degree in random geometric graphs with $n$ vertices. 
This provides a geometric counterpart of Montgomery's threshold result for binomial random graphs, and confirms a conjecture of Espuny D\'iaz, Lichev, Mitsche, and Wesolek. Our proof is algorithmic and adapts to other families of graphs, in particular graphs with bounded genus or tree-width.
\end{abstract}

\section{Introduction}
Finding particular structures in random graphs is a central topic in probabilistic combinatorics. The existence of a perfect matching~\cite{BF85,ER66}, a Hamilton cycle~\cite{AKS85,Bol84,BF85,KS83,Kor76,Pos76} or an $H$-factor \cite{AY93,JKV08,LR91,Ruc92} are classic properties which have been extensively studied and are by now quite well understood. An important aspect of this study has been the analysis of spanning trees in the binomial random graph $G(n,p)$ and, in particular, determining the thresholds for the containment of certain spanning trees~\cite{Gle13,HKS12,KLW16,Mon14b,Mon14a}.
A major breakthrough in this direction came with the work of Montgomery~\cite{Mon19} who showed that, for every $\Delta\ge 2$, there is $C = C(\Delta)$ such that, when $p\ge C(\log n)/n$, typically $G(n,p)$ contains copies of \emph{all} trees on $n$ vertices with maximum degree at most $\Delta$.
This property is often referred to as \emph{universality} for the containment of bounded-degree spanning trees.
It is widely believed~\cite{Gle13,Mon19} that the constant $C$ in the result of Montgomery should not depend on the degree $\Delta$, and that, in fact, the said result should hold for any constant $C>1$. Note that such a result would be optimal since $G(n,p)$ becomes connected around $p = (\log n)/n$. 
It is also worth mentioning the result of Johannesen, Krivelevich and Samotij \cite{JKS13} who showed that, for $\Delta = \Delta(n)$ suitably growing with $n$, every graph on $n$ vertices satisfying certain expansion criterion contains all $n$-vertex trees with maximum degree at most $\Delta$. In a related direction, Bhatt, Chung, Leighton and Rosenberg~\cite{BCLR89} constructed graphs on $n$ vertices and $O(n)$ edges containing all $n$-vertex trees with bounded maximum degree, and  graphs on $n$ vertices and $O(n\log n)$ edges containing all planar graphs with bounded degree.

In parallel with the developments in binomial random graphs discussed above, spanning structures in other classic random graph models were well studied.
Some of the most notable advances concern the \emph{random geometric graph} $G_d(n,r)$ where $n$ points are embedded in $[0,1]^d$ uniformly at random and independently, and a pair of points forms an edge if the Euclidean distance between them is at most $r$. 
A function $r_c = r_c(n)$ is said to be a \emph{sharp threshold} for a monotone property $P$ in the random geometric graph if, for every $\eps > 0$,
\[\lim_{n\to \infty} \mathbb P(G_d(n,r)\text{ satisfies }P) =
\begin{cases}
0, &\text{if } r = r(n)\le (1-\eps) r_c,\\
1, &\text{if } r = r(n)\ge (1+\eps) r_c.
\end{cases}\]
Sharp thresholds for the containment of perfect matchings~\cite{Pen99}, Hamilton cycles~\cite{BBKMW11,DMP07,MPW11} and balanced $n$-vertex trees~\cite{E-DLMW24} in the random geometric graph have been established. 

In this work, we establish a sharp threshold phenomenon for the containment of \emph{all} trees on $n$ vertices and bounded degree in $G_d(n,r)$.
In particular, we confirm \cite[Conjecture~14]{E-DLMW24} predicting the expression of the sharp threshold for $n$-vertex tree universality, and provide a geometric counterpart to the result of Montgomery~\cite{Mon19}.

Let $\mathcal{T}(n,\Delta)$ be the family of trees on $n$ vertices and maximum degree at most $\Delta$. We say that a graph $G$ is $\mathcal{T}(n,\Delta)$-universal if every $T\in \mathcal{T}(n,\Delta)$ is a subgraph of $G$ (not necessarily induced).

\begin{thm}\label{thm:1}
Fix $d\ge 1$ and $\Delta=\Delta(n)$ with $\Delta=n^{o(1)}$.
Then, the function 
\begin{equation}\label{eq:rc}
r_c = r_c(n, d, \Delta) := \frac{\sqrt{d}\log(\Delta-1)}{2\log n}
\end{equation}
is a sharp threshold for $G_d(n,r)$ to be $\mathcal{T}(n,\Delta)$-universal. 
More precisely, for all $\Delta=\Delta(n)$ as above,
\[\lim_{n\to \infty} \mathbb P(G_d(n,r)\text{ is }\mathcal{T}(n,\Delta)\text{-universal}) =
\begin{cases}
0, &\text{if } r = r(n)\le \left(1-\tfrac{\log (\Delta-1)}{\log n}\right) r_c,\\
1, &\text{if } r = r(n)\ge \left(1+\tfrac{100d\log (\Delta \log n) }{\log n}\right) r_c.
\end{cases}\]
\end{thm}

Let us note several important distinctions between \Cref{thm:1} and the result of Montgomery for binomial random graphs. 
First, the maximum degree in \Cref{thm:1} is allowed to grow with the size of the graph.
Second, \Cref{thm:1} exhibits a sharp threshold for universality: a phenomenon which has been conjectured but not yet confirmed for binomial random graphs.
Third, while universality holds for $p$ which is a constant factor away from the threshold for connectivity in $G(n,p)$, the threshold exhibited in \Cref{thm:1} is much larger than the connectivity threshold in $G_d(n,r)$, which is of order $\Theta((\log n)^{1/d}/n^{1/d})$~\cite{Pen03}.
This phenomenon is due to the different obstructions to embedding $n$-vertex trees: the binomial random graph is a good expander and the main obstruction there is given by vertices of small degree, whereas the random geometric graph exhibits good local clustering and large diameter.
In fact, combining the fact that the diameter of the host graph is dominated by the diameter of each of its spanning trees, and that $\cT(n,\Delta)$ contains a copy of a truncated $\Delta$-regular tree, suffices to deduce that the expression in~\eqref{eq:rc} is a lower bound for the threshold, see \Cref{sec:proof}.
We also observe that the condition $\Delta\ge 3$ in \Cref{thm:1} is necessary. Indeed, the threshold for Hamiltonicity of the random geometric graph asymptotically coincides with the threshold for connectivity mentioned above~\cite{BBKMW11,DMP07,MPW11}.

Our proof is algorithmic and is based on a tree-decomposition lemma which states that every tree with given maximum degree can be divided into a suitable number of trees of comparable sizes by removing a few edges from it (see \Cref{lem:divide}). 
By substituting \Cref{lem:divide} with results on balanced separation of graphs with suitably bounded tree-width or genus, our proof readily adapts to the latter two families. 
In the first case, we derive a sharp threshold for the property of finding all graphs on $n$ vertices with maximum degree at most $\Delta =n^{o(1)}$ and tree-width $t = t(n) = n^{o(1)}$.
In the second case, while being unable to provide a sharp threshold, we still ensure that typically all graphs on $n$ vertices with maximum degree at most $\Delta =n^{o(1)}$ and genus $g = g(n) \le n^{1-\eps}$ can be found simultaneously in $G_d(n,(1+o(1))r_c/\eps)$. For more details, see \Cref{sec: conc}.

\paragraph{Plan of the paper.}
In Section \ref{sec:proof}, we prove Theorem \ref{thm:1} by first establishing the lower bound, and then by showing the upper bound. In Section \ref{sec: conc}, we discuss our results and remaining open problems.

\paragraph{Notation.}
For a positive integer $n$, we write $[n] = \{1,\ldots,n\}$. Our asymptotic notation is standard and used exclusively with respect to $n$. We say that a sequence of events $(A_n)_{n}$ holds \textit{with high probability} if $\lim_{n\to\infty}\mathbb{P}(A_n)=1$. Rounding is omitted whenever it is irrelevant for the argument.

For a graph $G$, we write $V(G)$ for its vertex set, $E(G)$ for its edge set and $\diam(G)$ for its diameter.
Given a function $w:V(G)\to \mathbb R$ and a set $S\subseteq V(G)$, we write $w(S)=\sum_{v\in S}w(v)$.
We sometimes abuse notation and identify $G$ with its vertex set; in particular, $w(G)$ stands for $w(V(G))$.

Given a region $A\subseteq [0,1]^d$, we write $\vol_d(A)$ or simply $\vol(A)$ to denote the $d$-dimensional volume of $A$. 
For every centrally symmetric set $S\subseteq \mathbb R^d$ (such as a hypercube or a ball), we denote by $c(S)$ the centre point of $S$.
In what follows, $\left\|\cdot\right\|$ stands for the Euclidean distance (where the dimension is omitted for readability).

For the sake of clarity, we call the vertices of $G_d (n, r)$ landing in $[0,1]^d$ \emph{points} and reserve the term \emph{vertices} for the vertices of trees.

\section{\texorpdfstring{Proof of \Cref{thm:1}}{}}\label{sec:proof}

\paragraph{Lower bound.} Denote by $h = h(n,\Delta)$ the unique integer such that
\[\frac{(\Delta-1)^h-1}{\Delta-2} = \sum_{i=0}^{h-1} (\Delta-1)^i < n\le \sum_{i=0}^{h} (\Delta-1)^i = \frac{(\Delta-1)^{h+1}-1}{\Delta-2}.\]
In particular,
\begin{equation}\label{eq:bd h}
h\le \frac{\log((\Delta-2)n+1)}{\log(\Delta-1)}.
\end{equation}

Denote by $S$ an $n$-vertex rooted tree of maximum degree at most $\Delta$ and of height $h$ such that, for every $i\in [0,h-1]$, $S$ contains exactly $(\Delta-1)^i$ vertices at distance $i$ from its root. In particular, $\diam(S) \le 2h$.

Further, for every $d\ge 1$, the hypercube $[0,n^{-1/2d}]^d$ with volume $n^{-1/2}$ contains a vertex of $G_d(n,r)$ with probability $1 - (1-n^{-1/2})^n = 1-o(1)$, and the same holds for the hypercube $[1-n^{-1/2d},1]^d$.
In particular, for every $r\in (0,1]$, with high probability $\diam(G_d(n,r))\ge (1- 2n^{-1/2d})\sqrt{d}/r$.

Thus, combining the latter observation, the relation $\diam(G_d(n,r)) \le \diam(S) \le 2h$ and~\eqref{eq:bd h}, we obtain that, in order to find $S$ as a subtree of $G_d(n,r)$, with high probability, it is necessary that
\[(1-2n^{-1/2d})\frac{\sqrt{d}}{r}\le 2\frac{\log((\Delta-2)n+1)}{\log(\Delta-1)}\iff r\ge (1-2n^{-1/2d})\frac{\sqrt{d} \log(\Delta-1)}{2\log((\Delta-2)n+1)}.\]
Since the latter inequality implies that $r\ge (1-\tfrac{\log(\Delta-1)}{\log n})r_c$ for sufficiently large $n$, the lower bound follows. 

\paragraph{Upper bound.} We start with a lemma allowing us to divide vertex-weighted trees of bounded maximum degree into subtrees of comparable total weight by cutting through some of the edges. 
Although the proof idea is not new, see e.g.\ \cite[Proposition~4.5]{KN06}, we provide a short proof for completeness. 

\begin{lemma}\label{lem:divide}
Fix an integer $\Delta\ge 2$, real numbers $m > 0$ and $m_0 = m/(\Delta+1)$, a tree $T$ with maximum degree at most $\Delta$ and a weight function $w: V(T)\to (0, m_0]$ with $w(T)\ge m_0$. 
Then, there exist vertex-disjoint trees $T_1,\ldots,T_k\subseteq T$, such that $V(T)=\bigcup_{i\in [k]}V(T_i)$ and, for every $i\in [k]$, $w(T_i)\in [m_0,m]$.
\end{lemma}
\begin{proof}
Suppose towards contradiction that $T$ is a counterexample with the minimum number of vertices. Note that $m<w(T)$. In extension, for every $v\in V(T)$, we have that $w(v)+\Delta m_0\leq (\Delta+1)m_0=m < w(T)$.

Consider a vertex $v$ in $T$ which minimises $\max\{w(T'): T'\text{ is a connected component of }T\setminus \{v\}\}$. Further, denote by $T_v$ a tree achieving the said maximum. Let $u$ be the unique neighbour of $v$ in $T_v$, and let $T_u=T\setminus T_v$.
By the choice of $v$, the tree of maximal weight in the forest $T\setminus \{u\}$ must have weight at least $w(T_v)>w(T_v\setminus \{u\})$ and, therefore, it must be $T_u$. Thus, $w(T_v)\leq w(T_u)$. Moreover, since $T_v$ achieves the said maximum and since $v$ has at most $\Delta$ neighbours in $T$, we have that 
\[w(T_v)\ge \frac{w(T)-w(v)}{\Delta} \geq \frac{(w(v)+\Delta m_0)-w(v)}{
\Delta
}=m_0.\]
Thus, each of $T_u,T_v$ has weight strictly larger than $m_0$. As $T$ is a minimal counterexample, we have that, for each $T'\in\{T_u,T_v\}$, there exist disjoint trees $T_1',\ldots,T_k'\subseteq T'$ partitioning the vertices of $T'$ and such that, for every $i\in [k]$, $w(T_i')\in [m_0,m]$. Taking these two sets of trees together shows that $T$ satisfies the statement of  \Cref{lem:divide}. This gives a contradiction.
\end{proof}

Our analysis utilises the fact that the number of points in large regions of $[0,1]^d$ is well concentrated. 
This fact is a consequence of the following lemma for $p=1$; the purpose of the general statement will become clear later.
The proof of the lemma follows from a standard application of the Chernoff bound and is therefore omitted.

\begin{lemma}\label{lem:EquiDistributionOfPoints}
Fix a real number $p\in (0,1]$, an integer $n\ge 10/p$ and a set $A\subseteq [0,1]^d$ of volume $a\ge 10(np)^{-1}$. Consider a random set $P$ containing every point in $A\cap V(G_d(n,r))$ with probability $p$, independently for different points. 
Then, with probability at least $1-2\exp(-(anp)^{1/3}/3)$, $|P| \in [anp - (anp)^{2/3}, anp + (anp)^{2/3}]$.
\end{lemma}

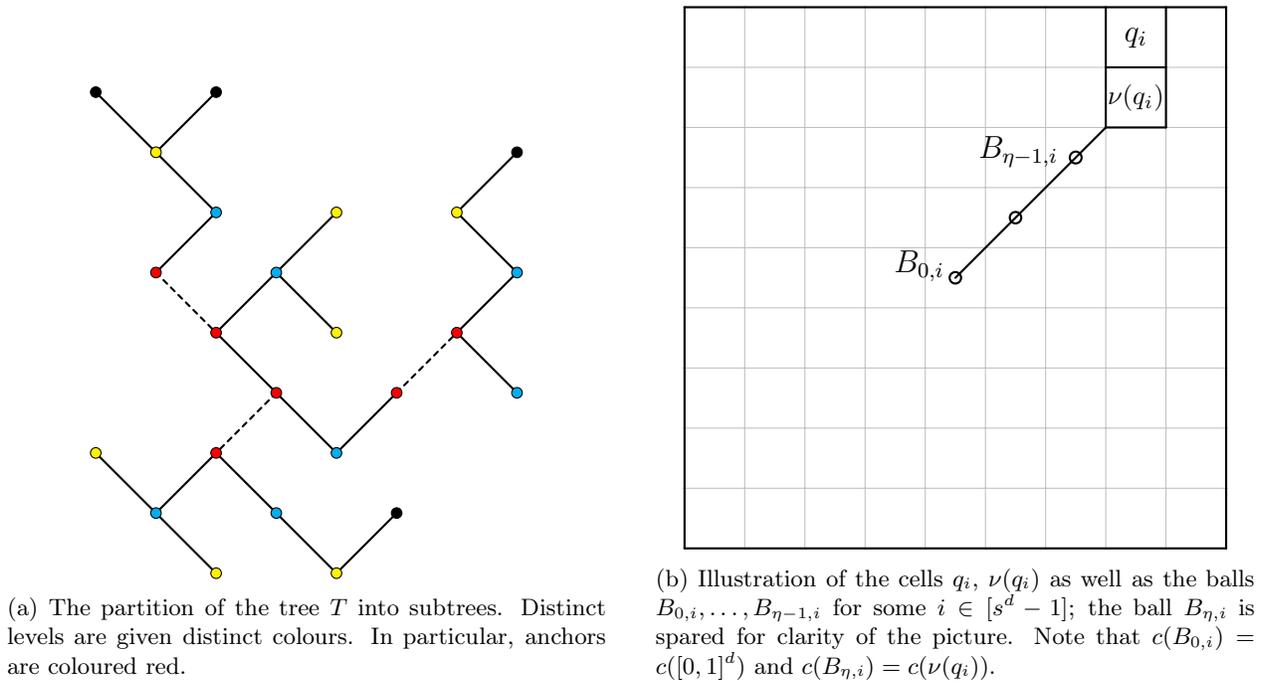
\begin{figure}
\centering
\begin{subfigure}[b]{0.48\textwidth}
\centering
\begin{tikzpicture}[scale=0.8,line cap=round,line join=round,x=1cm,y=1cm]
\clip(-7.2,-5.1) rectangle (0.2,3.1);
\draw [line width=0.8pt] (-7,3)-- (-6,2);
\draw [line width=0.8pt] (-6,2)-- (-5,3);
\draw [line width=0.8pt] (-6,2)-- (-5,1);
\draw [line width=0.8pt] (-5,1)-- (-6,0);
\draw [line width=0.8pt,dash pattern=on 2pt off 2pt] (-6,0)-- (-5,-1);
\draw [line width=0.8pt] (-5,-1)-- (-4,0);
\draw [line width=0.8pt] (-4,0)-- (-3,1);
\draw [line width=0.8pt] (-4,0)-- (-3,-1);
\draw [line width=0.8pt] (-5,-1)-- (-4,-2);
\draw [line width=0.8pt,dash pattern=on 2pt off 2pt] (-4,-2)-- (-5,-3);
\draw [line width=0.8pt] (-5,-3)-- (-6,-4);
\draw [line width=0.8pt] (-6,-4)-- (-5,-5);
\draw [line width=0.8pt] (-6,-4)-- (-7,-3);
\draw [line width=0.8pt] (-5,-3)-- (-4,-4);
\draw [line width=0.8pt] (-4,-4)-- (-3,-5);
\draw [line width=0.8pt] (-3,-5)-- (-2,-4);
\draw [line width=0.8pt] (-4,-2)-- (-3,-3);
\draw [line width=0.8pt] (-3,-3)-- (-2,-2);
\draw [line width=0.8pt,dash pattern=on 2pt off 2pt] (-2,-2)-- (-1,-1);
\draw [line width=0.8pt] (-1,-1)-- (0,-2);
\draw [line width=0.8pt] (-1,-1)-- (0,0);
\draw [line width=0.8pt] (0,0)-- (-1,1);
\draw [line width=0.8pt] (-1,1)-- (0,2);
\begin{scriptsize}
\draw [fill=black] (-7,3) circle (2.5pt);
\draw [fill=yellow] (-6,2) circle (2.5pt);
\draw [fill=black] (-5,3) circle (2.5pt);
\draw [fill=cyan] (-5,1) circle (2.5pt);
\draw [fill=red] (-6,0) circle (2.5pt);
\draw [fill=red] (-5,-1) circle (2.5pt);
\draw [fill=cyan] (-4,0) circle (2.5pt);
\draw [fill=yellow] (-3,1) circle (2.5pt);
\draw [fill=yellow] (-3,-1) circle (2.5pt);
\draw [fill=red] (-4,-2) circle (2.5pt);
\draw [fill=red] (-5,-3) circle (2.5pt);
\draw [fill=cyan] (-6,-4) circle (2.5pt);
\draw [fill=yellow] (-5,-5) circle (2.5pt);
\draw [fill=yellow] (-7,-3) circle (2.5pt);
\draw [fill=cyan] (-4,-4) circle (2.5pt);
\draw [fill=yellow] (-3,-5) circle (2.5pt);
\draw [fill=black] (-2,-4) circle (2.5pt);
\draw [fill=cyan] (-3,-3) circle (2.5pt);
\draw [fill=red] (-2,-2) circle (2.5pt);
\draw [fill=red] (-1,-1) circle (2.5pt);
\draw [fill=cyan] (0,-2) circle (2.5pt);
\draw [fill=cyan] (0,0) circle (2.5pt);
\draw [fill=yellow] (-1,1) circle (2.5pt);
\draw [fill=black] (0,2) circle (2.5pt);
\end{scriptsize}
\end{tikzpicture}
\caption{The partition of the tree $T$ into subtrees. Distinct levels are given distinct colours. In particular, anchors are coloured red.}
\label{fig:1a}
\end{subfigure}
\hfill
\begin{subfigure}[b]{0.48\textwidth}
\centering
\definecolor{cqcqcq}{rgb}{0.7529411764705882,0.7529411764705882,0.7529411764705882}
\begin{tikzpicture}[scale=0.8,line cap=round,line join=round,x=1cm,y=1cm]
\draw [color=cqcqcq, xstep=1cm,ystep=1cm] (-9.02,-4.02) grid (0.02,5.02);
\clip(-9.02,-4.02) rectangle (0.02,5.02);
\draw [line width=0.8pt] (-9,5)-- (-9,-4);
\draw [line width=0.8pt] (-9,-4)-- (0,-4);
\draw [line width=0.8pt] (0,-4)-- (0,5);
\draw [line width=0.8pt] (0,5)-- (-9,5);
\draw [line width=0.8pt] (-2,5)-- (-2,3);
\draw [line width=0.8pt] (-1,3)-- (-1,5);
\draw [line width=0.8pt] (-1,4)-- (-2,4);

\draw (-1.5,3.5) node {$\nu(q_i)$};
\draw (-1.5,4.5) node {\large{$q_i$}};

\draw (-5.1,0.7) node {\large{$B_{0,i}$}};
\draw (-4.9+1.45,1+1.6) node {\large{$B_{\eta-1,i}$}};

\draw [line width=0.8pt, opacity=0.3] (-4.5,0.5)-- (-2,3);
\draw [line width=0.8pt] (-4.5,0.5) circle (0.09527791532077659cm);
\draw [line width=0.8pt] (-3.5,1.5) circle (0.09527791532077659cm);
\draw [line width=0.8pt] (-2.5,2.5) circle (0.09850666049526213cm);
\draw [line width=0.8pt] (-2,3)-- (-1,3);
\end{tikzpicture}
\caption{Illustration of the cells $q_i$, $\nu(q_i)$ as well as the balls $B_{0,i},\ldots,B_{\eta-1,i}$ for some $i\in [s^d-1]$; the ball $B_{\eta,i}$ is spared for clarity of the picture.
Note that $c(B_{0,i})=c([0,1]^d)$ and $c(B_{\eta,i})=c(\nu(q_i))$.}
\label{fig:1b}
\end{subfigure}
\caption{Illustration from the proof of \Cref{thm:1}.}
\end{figure}

Let us first present the broad strategy of the proof of  the upper bound in \Cref{thm:1}. We embed a tree $T\in \mathcal{T}(n,\Delta)$ into $[0,1]^d$ as follows. We tessellate $[0,1]^d$ into congruent hypercubic cells of side length $\zeta r_c$ for a small constant $\zeta = \zeta(d)$, allowing points in the same cell and in adjacent cells to be connected. This creates a favourable local clustering structure. 
Furthermore, we partition $T$ into subtrees according to \Cref{lem:divide}. Note that \Cref{lem:divide} ensures that the sizes of these trees are comparable: we choose these sizes in a way guaranteeing that 
\begin{itemize}
    \item each subtree is sufficiently large so that, if we root it near the centre of $[0,1]^d$, then it can `stretch' and reach any point of $[0,1]^d$ and
    \item the expected number of points in each cell of the tessellation is multiple times larger than the size of any of the subtrees.
\end{itemize}

We call the endvertices of the edges in $T$ connecting different subtrees \emph{anchors}. For every subtree $R$, the vertices of $R$ are divided into \emph{levels} according to the distance of the vertex from its nearest anchor in $V(R)$. 
In particular, the anchors of $R$ belong to level $0$. This layering resembles that of a rooted tree but we may have multiple roots (see \Cref{fig:1a} for an illustration).

Our objective is then to use these subtrees to systematically populate the cells. To this end, we embed the trees one by one, each time aiming to embed the bulk of its vertices into a target cell that we choose for this tree. We prioritise cells which are further from the centre of $[0,1]^d$ since these are more difficult to reach from the anchors. In particular, cells near the corners of $[0,1]^d$ are the first to be fully populated. 

We embed a subtree $R$ and populate a target cell $q$ as follows.  First, we embed all the anchors in $V(R)$ near the centre of $[0,1]^d$. 
Then, we embed the vertices in the initial levels into points in designated, well-spaced balls that lie on a segment from centre of $[0,1]^d$ to a point near the target cell $q$. Having reached very close to this target cell, we embed the remaining vertices of $R$ into points in $q$. If there are not enough points in $q$ for this step, then we embed the surplus of vertices into an adjacent cell (see \Cref{fig:1b} for illustration, where the target cell depicted is $q=q_i$).

In what follows, we fix $r_c = r_c(n,d, \Delta)$ as defined in~\eqref{eq:rc} and assume throughout that $n$ is sufficiently large. Set
\begin{align}\label{eqn:epsilon}
    \eps = \eps(n,d,\Delta) := \frac{100 d\log(\Delta\log n)}{\log n} \quad\text{and}\quad r  \geq   (1+\eps)r_c.
\end{align}

\begin{proof}[Proof of the upper bound in \Cref{thm:1}]
Fix $n$ suitably large.
Let $\cQ$ be the tessellation of $[0,1]^d$ into $s^d$ congruent (closed) hypercubic cells, where $s$ is an odd integer such that 
\begin{align*}
    \frac{ \sqrt{ d } }{ s } \in \left[\frac{1}{3}\left(1+\frac{\eps}{2}\right)r_c, \frac{1}{2}\left(1+\frac{2\eps}{3}\right)r_c\right].
\end{align*}
While some cells in $\cQ$ intersect along $k$-dimensional faces for some $k\le d-1$, almost surely none of the $n$ embedded points lands on the boundary between two cells. For this reason, we systematically assume that the cells induce a partition of the point set.

We first collect some notation regarding the tessellation which will be of use during the embedding.
We order the cells of $\cQ$ as a sequence $q_1, \dots, q_{s^d}$ so that the distance from the centre of each cell to $c ([0, 1] ^ d)$ is non-increasing. In this ordering, $q_{s^d}$ denotes the central cell, that is, the unique cell that contains the point $(1/2,\ldots, 1/2)$. 
For every cell $q_i \neq q_{s^d}$, we let $j_i$ be minimum index such that the $j_i$-th coordinates of $c(q_i)$ and $c(q_{s^d})$ disagree. Then, we define the \emph{adjacent successor of $q_i$} to be the cell $\nu(q_i)\in \cQ$ whose centre $c(\nu(q_i))$ differs from $c(q_i)$ in the $j_i$-th coordinates only, $\|c(\nu(q_i))-c(q_i)\|=1/s$ and $c(\nu(q_i))$ is closer to $c([0,1]^d)$ than $c(q_i)$. 
Thus, $q_i$ precedes $\nu(q_i)$ in the ordering $q_1, \dots, q_{s^d}$, and 
\begin{align}\label{prop:distance_among_cells}
\text{for every pair of points $x\in q_i$ and $y\in \nu(q_i)$, we have that $\|x-y\|\leq \sqrt{\frac{4}{s^2}+\frac{d-1}{s^2}} \leq \frac{2\sqrt{d}}{s}\le r$.}
\end{align}

Fix $\eta:= \left\lceil s/4\right\rceil$. Note that $\Delta=n^{o(1)}$ implies that $\eps=o(1)$ and, in extension, that $\sqrt{d}/2 \le \eta \cdot  (1-\eps/4)$.
For every $j\in \{0,...,\eta\}$, denote by $B_{j,i}'$ the ball of radius $\eps/(10s)$ centred at 
\[c ([0,1] ^ d) + (j/\eta)\cdot\left(c (\nu(q_i)) - c ([0,1] ^ d)\right).\] 
Note that, for every $i,j$ as above, $B_{j,i}'$ may intersect up to $2^d$ cells. Hence, at least one of these cells contains a ball of radius $2^{-d}\eps/ (10s)$, say, which is included in $B_{i,j}'$ and whose centre lies on the segment from  $c ([0,1] ^ d)$ to $c (\nu(q_i))$. We let $B_{j,i}$ be such a ball. (See Figure \ref{fig:1b} for illustration).

Note that the balls $B_{0, i}, B_{1, i}, \dots, B_{\eta, i}\subseteq [0,1]^d$ satisfy the following properties:
\begin{enumerate}[\textbf{P1}]
    \item \label{prop:balls_in_cells_after}      
    $\bigcup_{j = 0}^{\eta} B_{j, i} \subseteq \bigcup_{j = i+1}^{s ^ d}q_j$;
\end{enumerate}
\begin{enumerate}[\textbf{P2}]
    \item The balls $B_{0, i}$ and $B_{\eta,i}$ are contained in $q_{s^d}$ and $\nu(q_i)$, respectively;
    \label{prop:1st_last_balls_location}
    \end{enumerate}
\begin{enumerate}[\textbf{P3}]
    \item For every $j\in [0,\eta-1]$ and every pair of points $x\in B_{j,i}$ and $y\in B_{j+1,i}$, we have that $\|x-y\|\leq r$. \label{prop:distance_among_balls}
\end{enumerate}
Properties \ref{prop:1st_last_balls_location} and \ref{prop:distance_among_balls} are immediate by construction, and the observation that  $\sqrt{d}/2 \le \eta \cdot  (1-\eps/4)$ implies that for every point $v\in [0,1]^d$ we have that $(\|c ([0,1] ^ d)-v\|)/\eta \leq (1-\eps/4) r$. As for Property \ref{prop:balls_in_cells_after}, let $q_j$ be the cell containing $B_{j,i}$. If $q_j\neq \nu(q_i)$, then, as $q_j$ contains the centre of $B_{j,i}$, $\|c(q_j)-c(B_{j,i})\|\le \|c(\nu(q_i))-c(B_{j,i})\|$. As both $c(B_{j,i})$ and $c (\nu(q_i))$ belong to the segment from $c ([0,1] ^ d)$ to $c (\nu(q_i))$, by the triangle inequality
\begin{align*}
\|c([0,1]^d)-c(q_j)\|&\le \|c([0,1]^d)-c(B_{j,i})\|+\|c(B_{j,i})-c(q_j)\|\\
&\le \|c([0,1]^d)-c(B_{j,i})\|+\|c(B_{j,i})-c(\nu(q_i))\|=\|c([0,1]^d)-c(\nu(q_i))\|.
\end{align*}
Hence, $q_j$ succeeds $\nu(q_i)$ in the ordering $q_1,\dots, q_{s^d}$, which implies property \ref{prop:balls_in_cells_after}. 

We generate $G_d(n,r)$ and colour each of its points independently, in blue with probability $1/2$ and in red otherwise.
Let $\mathcal{A}$ be the event that each of the following properties are satisfied.
\begin{enumerate}[\textbf{A1}]
    \item \label{A1} 
    For every $i \in [s^d - 1]$ and $j \in [\eta]$, the ball $B_{j, i}$ contains at least $\vol(B_{j,i})n/4\ge d^{-d/2} (\eps/(2^d\cdot 10 s))^d\cdot n/4$ red points, where we used that the volume of a unit $d$-dimensional ball is at least $d^{-d/2}$.
\end{enumerate}
\begin{enumerate}[\textbf{A2}]
    \item \label{A2} 
    For every $i \in [s^d]$,  the cell $q_i$ contains at least $\frac{ 3 }{ 8 } s ^ {- d} n $ blue points.
\end{enumerate}

Note that, by \Cref{lem:EquiDistributionOfPoints} and a union bound, properties \ref{A1} and \ref{A2} hold jointly with high probability. In the remainder of the proof, we show that, conditionally on the event $\mathcal{A}$, $G_d (n, r)$ contains all trees on $n$ vertices with maximum degree at most $\Delta = \Delta(n)$. 

Fix any tree $T\in \mathcal{T}(n,\Delta)$. Applying \Cref{lem:divide} with $m = m(n) := s ^ {- d} n / (8 d)$ and uniform vertex weight function $w \equiv 1$ on $V(T)$, we obtain a set of subtrees $T_1,...,T_k$ of $T$, such that $V(T_1),...,V(T_k)$ partitions $V(T)$ and each subtree has between $m_0 = m/(\Delta+1)$ and $m$ vertices.

We call a vertex in $V(T)$ an \emph{anchor} if it is incident to one of the $k-1$ edges in $E(T)\setminus \left(\bigcup_{\ell=1}^k E(T_\ell)\right)$. Thus, the total number of anchors is  at most  $2 k - 2$. For a vertex $v \in V (T_j)$, define its \emph{level} to be the minimum distance from $v$ to an anchor of $T_j$. (See Figure \ref{fig:1a} for illustration).

To embed $T$ in $G(n,r)$, we execute the embedding algorithm given below. While executing the embedding algorithm, we call a point of $G(n,r)$ \textit{occupied} if we have already embedded some vertex of $T$ into it.  If at any step of the algorithm there are not enough unoccupied points for it to embed a prescribed set of vertices, then the algorithm terminates and returns FAILURE; we will argue that, conditionally on $\cA$, this does not happen. Else, the algorithm terminates upon running the for-loop for $k$ iterations.   
\vspace{3mm}
\\\noindent \textbf{Embedding Algorithm:} For $t=1,...,k$, do the following.
Let $i_t$ be the minimal index such that there is an unoccupied point in $q_{i_t}$. If $i_t = s^d$, then embed all vertices of the tree $T_t$ into points in the cell $q_{s^d}$. Else, execute Step 1 and then Step 2.

\vspace{3mm}
\noindent    
\textbf{Step 1.} 
 For $j=0,...,\eta$, embed all vertices at layer $j$ of $T_t$ into red points in $B_{j,i_t}$.

\vspace{3mm}
\noindent   
\textbf{Step 2.}  Embed the rest of the  vertices in $q_{i_t}\cup \nu(q_{i_t})$ by using first the unoccupied points in $q_{i_t}$, and then, if needed, any blue points in $\nu(q_{i_t})$. 

\vspace{3mm}

Note that, if $uv\in E(T)$ and both vertices are embedded into points $p(u),p(v)\in V(G_d(n,r))$ before the algorithm terminates, then one of the following holds. 
\begin{itemize}
    \item $u,v$ belong to distinct subtrees. Then, each of $u$ and $v$ is an anchors. In this case, \ref{prop:1st_last_balls_location} ensures that $p(u),p(v)\in \big(\bigcup_{t=1}^k B_{0,t}\big)\cup q_{s^d}\subseteq q_{s^d}$.
    \item $u,v$ belong to some tree $T_t$ with $t\in [k]$, and $i_t=q_{s^d}$. Then, each of $u$ and $v$ is embedded into a point in $q_{s^d}$. Thus, $\|p(u)-p(v)\|\leq r$.
    \item $u,v$ belong to some tree $T_t$ with $t\in [k]$, and $i_t\neq q_{s^d}$. Furthermore, there exists $j\geq 0$ such that either each of $u$ and $v$ belongs to level $j$ of $T_t$, or one of them belongs to level $j$ and the other one to level $j+1$ of $T_t$. 
    \begin{itemize}
        \item If $j<\eta$, then $p(v),p(u)\in B_{j,i}\cup B_{j+1,i}$. In particular, $\|p(u)-p(v)\|\leq r$, by \ref{prop:distance_among_balls}.
        \item Else $p(u),p(v)\in B_{\eta,i}\cup q_i\cup \nu(q_i)= q_i\cup \nu(q_i)$, where the equality holds by \ref{prop:1st_last_balls_location}. Then, $\|p(u)-p(v)\|\leq r$ by \eqref{prop:distance_among_cells}.
    \end{itemize}
\end{itemize}
In all cases, the inequality $\|p(u)-p(v)\|\leq r$ holds and, therefore, $p(u)p(v) \in E(G_d(n,r))$. 

It remains to prove that, conditionally on the event $\mathcal{A}$, the embedding algorithm does not return FAILURE and, as a result, embeds $T$ in $G_{d}(n,r)$. 
By recalling that the number of vertices of $T$ equals the number of points, the only (potential) reasons for FAILURE are analysed in the following two cases.  

\vspace{0.5em}
\noindent
\textbf{Case 1:} There exist $t\in [k]$ and $j\in \{0,\ldots, \eta\}$ such that, at the $t$-th iteration of the for-loop, the number of vertices at level $j$ of $T_t$ is strictly larger than the number of unoccupied red points left in $B_{j,i_t}$.
\vspace{3mm}
\\ 
\noindent
Then, \ref{prop:balls_in_cells_after} implies that, during the first $t$ iterations, the embedding algorithm will embed into red points in $B_{j,i_t}$ only vertices which are at some level $j'$ of some tree $T_{t'}$ with $j'\leq \eta$ and $t' \leq t$. 
Every such vertex is at graph distance (in $T$) at most $\eta$ from an anchor of $T$, that is, at level at most $\eta$. We denote the number of these vertices by $b$. Recall that $r_c=(\sqrt{d}\log(\Delta-1))/(2\log n) $. First,
\[k\leq n/m_0 = 8d(\Delta +1)s^d \leq 8d(\Delta +1) \left( \frac{2\sqrt{d}}{\left(1+\frac{\eps}{3}\right)r_c}\right)^d \leq 8d(\Delta +1) \left(\frac{4\log n}{\left(1+\frac{\eps}{3}\right)\log(\Delta-1)}\right)^d \leq \Delta (\log n)^{2d}.\]
Thereafter, since $T$ has maximum degree $\Delta$, and at most $2k-2$ anchors,
\begin{align*}
    b&\le (2k-2)\sum_{i=0}^{\eta}(\Delta-1)^{i} \le 2(2k-2)(\Delta-1)^\eta\le 4\Delta(\Delta-1)^{\eta}(\log n)^{2d}.
\end{align*}
Recalling that $\eta=  \left\lceil s/4\right\rceil$ 
and $s\le \frac{2\sqrt{d}}{(1+\eps/3)r_c}$, we have
\begin{align*}
    b&\le 4\Delta^2 (\Delta-1)^{s/4} (\log n)^{2d} \le 4\Delta^2 (\Delta-1)^{\sqrt{d}/(2(1+\eps/3)r_c)}(\log n)^{2d} =4\Delta^2 n^{1/(1+\eps/3)}(\log n)^{2d}.
\end{align*}
Furthermore, \eqref{eqn:epsilon} implies that $n^{\eps}=\Delta^{100d}(\log n)^{100d}$, and the condition $\Delta=n^{o(1)}$ implies that $\epsilon=o(1)$. Therefore, 
\begin{align*}
    b&\le 4\Delta^2 n^{1-\eps/4}(\log n)^{2d} = 4\Delta^2 n \cdot\Delta^{-25d}(\log n)^{-25d} (\log n)^{2d} \leq  \Delta^{-25d}(\log n)^{-25d}\cdot n.
\end{align*}
Since 
\begin{align*}
    \frac{1}{d^{d/2}}\left(\frac{\eps}{2^d\cdot 10s}\right)^d\cdot \frac{n}{4}&\ge\left(\frac{\eps\cdot r_c}{2^d\cdot 10 d}\right)^d\cdot \frac{n}{4} \ge \left(\log n\right)^{-2d}\cdot \frac{n}{4}> \Delta^{-25d}(\log n)^{-25d}\cdot n \geq b,
\end{align*}
the impossibility to perform the embedding step in the ball $B_{j,i_t}$ implies that the event \ref{A1} does not occur, leading to a contradiction.

\vspace{0.5em}
\noindent
\textbf{Case 2:} there exists $t\in [k]$ such that, at iteration $t$, the number of vertices of $T_t$ at distance at least $\eta+1$ from the anchors in $T_t$ is strictly larger than the union of the still unoccupied points in $q_{i_t}$ and the blue points in $\nu(q_{i_t})$.
\vspace{3mm}
\\ \noindent
Note that, during the first $t$ iterations, the only vertices embedded into blue points in $\nu(q_{i_t})$ by the embedding algorithm are ones which were not fully embedded in a cell $q_{t'}$ with $t' < t$ with adjacent successor of $\nu(q_{i_{t'}})=\nu(q_{i_t})$.
In particular, since every cell is the adjacent successor of at most $2d$ cells, there are at most $2d-1$ such trees. 
Thus, if Case~2 occurs, then the number of blue points in $\nu(q_i)$ is at most $(2d-1) m + m = s^{-d}n/4$, thus contradicting~\ref{A2}.
This concludes the proof of \Cref{thm:1}.
\end{proof}

\section{Conclusion}\label{sec: conc}

In this note, we establish a sharp threshold for the property of the random geometric graph $G_d(n,r)$ being $\mathcal{T}(n,\Delta)$-universal with $\Delta\ge 3$.
\Cref{thm:1} handles the random geometric graph defined on the unit hypercube $[0,1]^d$ but, with relatively minor modifications, our approach may be adapted to other natural ambient spaces such as the unit $d$-dimensional torus, the unit $d$-dimensional ball or sphere. 
Such modifications include, in particular, adjusting the leading constant in $r_c$ defined in~\eqref{eq:rc}.

While \Cref{thm:1} provides bounds for the critical window where the probability of the said property transitions from $o(1)$ to $1-o(1)$, it would be interesting to analyse this phase transition more closely.
More ambitiously, one could try to determine a simple hitting time criterion which is asymptotically equivalent to the studied universality phenomenon.

Another related direction of study concerns the threshold for containment of \emph{almost} all $n$-vertex trees in the random geometric graph. The answer to this question follows from known results when $d=1$, as shown in the following proposition.

\begin{proposition}\label{prop:1}
Fix a uniform random tree $T$ on $n$ vertices.
If $r = o(n^{-1/2})$, with high probability $G_1(n,r)$ contains no copy of $T$. 
If $r = \omega(n^{-1/2})$, with high probability $G_1(n,r)$ contains a copy of $T$. 
\end{proposition}
\begin{proof}[Sketch of proof.]
Root the tree $T$ at a random vertex. Then, by \cite[Section~4.2.7]{Dra09}, for every $\eps\in (0,1]$ and suitably large $n$, there exists $\delta = \delta(\eps)\in (0,1]$ such that, with probability at least $1-\eps$, the height of $T$ lies in the interval $[\delta \sqrt{n}, \delta^{-1}\sqrt{n}]$.
The first statement in the proposition follows from the comparison of the diameters of $T$ and $G_1(n,r)$ as in the proof of the lower bound in \Cref{thm:1}.

At the same time, recall that the \emph{width} of a rooted tree $T$ is the maximum over $t\ge 0$ of the number of vertices at distance $t$ from the root.
Then, by \cite[Section~4.2]{Dra09}, for every $\eps\in (0,1]$ and suitably large $n$, there is $\delta = \delta(\eps)\in (0,1]$ such that, with probability at least $1-\eps$, $T$ has width in the interval $[\delta \sqrt{n}, \delta^{-1}\sqrt{n}]$.
Thus, when $r = \omega(n^{-1/2})$, the latter observation and the equi-distribution of the vertices of $G_1(n,r)$ (see \Cref{lem:EquiDistributionOfPoints}) ensure that the tree $T$ can be embedded in $G_1(n,r)$ by exploring $T$ according to a breadth-first search and greedily embedding every new vertex of $T$ at the left-most vertex of $G_1(n,r)$ which is still unoccupied. This implies the second part of the proposition.
\end{proof}

We note that the embedding strategy from \Cref{prop:1} can be adapted to $G_d(n,r)$ when $r = \omega(n^{-1/2d})$ by tessellating $[0,1]^d$ into hypercubes of side length $s$ with $\omega(n^{-1/2d}) = s = o(r)$ and embedding the random tree greedily along a path of consecutive hypercubes.
While this strategy already vastly improves over the threshold in \Cref{thm:1} (where we use that $\Delta(T) = O(\log n)$ for a uniformly random tree $T$, see e.g.\ \cite[Section~3.2.1]{Dra09}), the obstruction given by the diameter of $T$ when $d\ge 2$ becomes weaker than the obstruction given by the absence of isolated vertices.
Thus, we believe that the threshold for embedding a random tree (and, as a consequence, almost all trees) on $n$ vertices in $G_d(n,r)$ may coincide with the threshold for connectivity.
Confirming or refuting this statement remains an intriguing open problem.

\vspace{0.5em}

We also observe some extensions of \Cref{thm:1} for more general graph classes. For each of them, the only component of the proof requiring modifications is \Cref{lem:divide}.
First, the exhibited sharp threshold result can be transferred to $n$-vertex graphs $G$ with maximum degree at most $\Delta = n^{o(1)}$ and tree-width $t = t(n) = n^{o(1)}$ by replacing \Cref{lem:divide} with (an iterative application of) a result of Robertson and Seymour~\cite{RS86}\footnote{While this work may be hard to access freely, an explanation of the result may be found in~\cite{DN19}.} ensuring that every such graph $G$ contains disjoint subgraphs $G_1,G_2$ not connected by an edge, with at most $2n/3$ vertices each and such that $G\setminus (G_1\cup G_2)$ contains at most $t+1$ vertices.
Second, for fixed $\eps > 0$ and any $n$-vertex graph $G$ with genus $g = g(n) \le n^{1-2\eps}$, Gilbert, Hutchinson and Tarjan~\cite{GHT84} showed that $G$ can be divided into subgraphs $G_1,G_2$ as above where $G\setminus (G_1\cup G_2)$ contains $O(\sqrt{gn})$ vertices. 
While the latter result does not translate into a sharp threshold using our approach since the number of anchors can be as large as $n^{1-\eps+o(1)}$, our approach ensures that all graphs $G$ with maximum degree at most $\Delta = n^{o(1)}$ and genus at most $g$ are simultaneously contained in $G_d(n,r)$ when $r$ is slightly above $r_c/\eps$ with high probability.
Understanding if the universality problem for this family of graphs admits a sharp threshold remains an interesting direction for further research.

\section*{Acknowledgement}
The authors wish to thank Michael Krivelevich for his remarks on an earlier version of the paper. This research was funded in part by the Austrian Science Fund (FWF) grants No. 10.55776/ESP624 and No. 10.55776/ESP3863424. 
Part of this work was conducted while the second and fifth authors were visiting TU Wien and they thank the institution for its hospitality. 
For open access purposes, the authors have applied a CC BY public copyright license to any author accepted manuscript version arising from this submission. 

\bibliographystyle{abbrv}
\bibliography{Bib}

\end{document}